\documentclass[11pt]{article}
\usepackage{amsmath,amsthm,amssymb,fancybox,enumerate,color}
\usepackage{latexsym,amssymb}

\usepackage{mathrsfs,color}

\newtheorem{theorem}{Theorem}[section]

\newtheorem{proposition}[theorem]{Proposition}

\theoremstyle{definition}

\theoremstyle{remark}
\newtheorem{remark}[theorem]{Remark}

\numberwithin{equation}{section}

\def\re{\mathbb{R}}

\def\ol{\overline}

\def\|{\Vert}

\def\({\left(}
\def\){\right)}
\def\e{\varepsilon}

\def\a{\alpha}

\def\intb{\int_B}
\def\a{\alpha}

\begin{document}

\title{%
$W^{1,p}$ approximation of the Moser--Trudinger inequality
}

\author{Masato Hashizume\\
Graduate School of Advanced Science and Engineering, Hiroshima University, 
\\
Higashihiroshima, 739-8527, Japan
\\[5mm]
Norisuke Ioku
\\
Mathematical Institute, Tohoku University, 
\\
Aramaki 6-3, Sendai 980-8578, Japan
}

\date{}

\maketitle

\begin{abstract}
We propose a power type approximation of the Moser--Trudinger functional and show that its concentration level converges to the Carleson--Chang
limit.
\end{abstract}
\maketitle

\section{Introduction}
Let $N \geq 2$, $1< p<N$, $B$ be the unit ball in 
$\re^N$, 
and 
$W^{1,p}_0(B)$ be the completion with respect to the Sobolev norm 
$\|u\|_{W^{1,p}(B)}=(\|u\|_{L^p(B)}^p+\|\nabla u\|_{L^p(B)}^p)^{1/p}$
 of smooth compactly supported functions in $B$.
Then, the Sobolev inequality in $B$ 
states that,
there exists a constant $C>0$ such that 
\begin{equation}\label{eq:1.1a}
\|u\|_{L^{p^*}(B)}
 \le 
 C
 \|\nabla u\|_{L^p(B)}
\end{equation}
for every $u\in W^{1,p}_0(B)$,
where $p^* = Np/(N - p)$ is the critical Sobolev exponent.

In the borderline situation where $p = N$, the inequality \eqref{eq:1.1a} is known to hold when $p^*$ is replaced by any number greater than or equal to $1$. However, a stronger result, proved by Trudinger~\cite{Trudinger} (see also \cite{Pohozaev,Yudovich}) is available. 
That is,
\begin{equation}\label{TM}
\sup_{\substack{u \in W^{1,N}_0(B) \\ \|\nabla u\|_{L^N(B)} \leq 1}} \int_B e^{\a |u|^\frac{N}{N-1}} dx \leq C|B|
\end{equation}
for some constants $\alpha$ and $C$, depending only on $N$.
Moser \cite{Moser} sharpened the inequality~\eqref{TM} to
\begin{equation}
\label{originalTM}
\sup_{\substack{u \in W^{1,N}_0(B) \\ \|\nabla u\|_{L^N(B)} \leq 1}} \int_B e^{\a |u|^{\frac{N}{N-1}}} dx
\begin{cases}
\le C|B| \quad &{\rm if}\quad \a \leq \a_N, \\
= +\infty &{\rm if}\quad \a > \a_N,
\end{cases}
\end{equation}
where $\a_N = N \omega_{N-1}^{\frac{1}{N-1}}$ and $\omega_{N-1}$ is the surface measure of the unit sphere 
in $\re^N$.

The Trudinger inequality~\eqref{TM} 
is considered to be a limiting case of the Sobolev inequality 
in the framework of Orlicz spaces. 
After the contribution of Adams~\cite{Adams},
Cianchi~\cite{Cianchi} established the optimal extension of inequalities \eqref{eq:1.1a} and \eqref{TM} to the case where Lebesgue norms are replaced by any Orlicz norm. This extension coincides with 
\eqref{eq:1.1a} for $W^{1,p}_0(B)$ and \eqref{TM} for $W^{1,N}_0(B)$, respectively.
However, the latter is not obtained via a direct limiting procedure in the former as $p\to N$
even though the Sobolev inequality~\eqref{eq:1.1a} and its limiting case~\eqref{TM}
were unified in \cite{Cianchi}.

In this paper, we focus on this discontinuity and propose 
an equivalent form of the $L^{p^*}$ norm that converges to 
the Moser--Trudinger functional in ~\eqref{originalTM}
as $p\to N$.
To state our result, we define a function $F_p: \re \to \re_+$ for $p \in (1,N)$ by
\begin{equation}
\label{F_p}
\left\{
\begin{aligned}
F_p(s) 
&
:= \left[1 + \frac{N-p}{N(p-1)}\a_p | s |^{\frac{p}{p-1}}
\right]^{\frac{N(p-1)}{N-p}},
\\
\a_p 
&
:= 
\left(
\a_N^{\frac{N-1}{N}} |B|^{\frac{1}{p}-\frac{1}{N}}
\right)^{\frac{p}{p-1}}.
\end{aligned}
\right.
\end{equation}
\begin{proposition}\label{proposition:1}
Let $u$ be a smooth compactly supported function in $B$.
Then there holds
\[
c_1 
\( \|u\|_{L^{p^*}(B)}^{p^*} + |B| \)
\le
\int_B F_p(u)dx
\le
c_2 
\( \|u\|_{L^{p^*}(B)}^{p^*} + |B| \)
\]
for some $c_1,c_2>0$ depending on $p$, $N$, and $|B|$.
Furthermore, there holds
\[
\int_{B}F_p\bigl(u
\bigr) dx
\to
\int_B e^{\a_N |u|^{\frac{N}{N-1}}} dx\ \ (p\to N).
\]
\end{proposition}
Proposition~\ref{proposition:1} clearly follows from
$\lim_{s\to \infty}F_p(s)/s^{p^*}=c$ for some constant $c>0$
and $\lim_{s\to 0}F_p(s)=1$.
Furthermore, as we will see, 
the proposed function $F_p$ 
yields new insight into the concentration level of the Moser--Trudinger functional in~\eqref{originalTM}.

The concentration level of the Moser--Trudinger functional~\eqref{originalTM} 
for $\a=\a_N$  
was first investigated 
by Carleson--Chang~\cite{CC}.
Let 
$B_{\e}$ 
be
the ball centered at the origin with radius $\e>0$
and
\[
G(u) := \int_B e^{\a_N |u|^{\frac{N}{N-1}}}dx.
\]
It is revealed in~\cite{CC} that
\begin{equation}
\label{CClimit}
\limsup_{n \to \infty}G(u_n) \leq |B|(1+e^{\sum_{k=1}^{N-1}\frac{1}{k}})
\end{equation}
if $\{u_n\}$ is a concentrating sequence, that is $\lim_{n \to \infty} \| \nabla u_n\|_{L^N(B_\e)} = 1$ holds for every $\e>0$.
The maximal limit in \eqref{CClimit} on concentrating sequences is called the Carleson--Chang limit.
Later, de Figueiredo--do \'O--Ruf \cite{FOR} constructed a concentrating sequence $\{y_n\}$ such that $\lim_{n \to \infty} G(y_n) = |B|(1+e^{\sum_{k=1}^{N-1}\frac{1}{k}})$, 
which means the value of the Carleson--Chang limit is the right hand side of \eqref{CClimit}.
It should 
be mentioned that
Carleson--Chang~\cite{CC}
considered \eqref{CClimit}
to prove
the existence of a function which attains the supremum in (\ref{originalTM}) for $\a=\a_N$.
Specifically, they
described a function $u^*$ such that $G(u^*) > |B|(1+e^{\sum_{k=1}^{N-1}\frac{1}{k}})$,
and combined this fact,
the concentration compactness argument, and \eqref{CClimit} 
to show that
all maximizing sequences of the supremum in \eqref{originalTM} with 
$\alpha=\alpha_N$ are
precompact.
This method has been extended to more general cases, as shown in the works of Struwe \cite{Struwe}, Flucher \cite{Flucher} and Li \cite{Li}.
See also \cite{Ruf,Ishiwata,IMNS,MM,Hashizume}
and references therein for other discussion of maximizing problems related to the Moser--Trudinger functional.

We define $\mathscr{B}_p$ and $X_p$ by
\begin{equation}
\label{B_p}
\mathscr{B}_p := \left\{ u \in W^{1,p}_{0,rad}(B) \ \middle|\ \|\nabla u\|_p \leq 1\right\},
\end{equation}
where $W^{1,p}_{0,rad}(B)$ denotes
the set of radially symmetric functions 
belonging to $W^{1,p}_0(B)$,
and
\begin{equation}
\label{X_p}
X_p := \left\{\{u_n\} \subset \mathscr{B}_p \ \middle|\ \lim_{n \to \infty} \|\nabla u_n \|_{L^p(B_\e)} = 1\ {\rm for \ any\ }\e>0 \right\}.
\end{equation}

The main result of the presented paper shows that 
the concentration level associated with $F_p$ converges to 
the Carleson--Chang limit.
\begin{theorem}
\label{thm1}
It holds that
\[
\sup_{\{u_n\} \in X_p}
\left(
\limsup_{n \to \infty} \int_B F_p(u_n)dx
\right)
\to |B|(1+e^{\sum_{k=1}^{N-1} \frac{1}{k}})\ \ (p\to N).
\]
\end{theorem}
\begin{remark}
It is important to notice the derivation of $F_p$.
In the proof of \eqref{CClimit}, 
the following inequality, termed Alvino's inequality~\cite{Alvino}
or the radial lemma, plays an essential role:
\begin{equation}
\label{Alvino_ineq}
|u(x)|
\le 
\alpha_N^{-\frac{N-1}{N}
}\(N\log \frac{1}{|x|}\)^{\frac{N-1}{N}}\|\nabla u\|_{L^N(B)}
\end{equation}
for every $x\in B\setminus\{0\}$ and $u\in W^{1,N}_{0,rad}(B)$.
It is easy to check that \eqref{Alvino_ineq} is equivalent to
\[
e^{\alpha_N\left(|u(x)|/\|\nabla u\|_{L^N(B)}\right)^{\frac{N}{N-1}}}\le \frac{1}{|x|^N}.
\]
Similarly, for the case of $1<p<N$, 
it holds that
\begin{equation}
\label{pradial}
|u(x)|\le \left(|B|^{\frac{1}{p}-\frac{1}{N}}\alpha_N^{\frac{N-1}{N}}\right)^{-1}
\left\{
N\frac{p-1}{N-p}\left(|x|^{-\frac{N-p}{p-1}}-1\right)
\right\}^{\frac{p-1}{p}}\|\nabla u\|_{L^p(B)}
\end{equation}
for every $x\in B\setminus\{0\}$ and $u\in W^{1,p}_{0,rad}(B)$.
The function $F_p$ is defined 
so that
\eqref{pradial}
is equivalent to 
\[
F_p\bigl(u(x)/\|\nabla u\|_{L^p(B)}\bigr)\le \frac{1}{|x|^N}.
\]
We prove \eqref{pradial} at the end of Section~2 for the convenience of readers. 
\end{remark}

%%%%%%%%%%%%%%%%%%%%%%%%%%%%%%%%%%%%%%%%%%%%%%%%%%%%%%%%%%%%%
%%%%%%%%%%%%%%%%%%%%%%%%                      %%%%%%%%%%%%%%%%%%%%%%%%
%%%%%%%%%%%%%%%%%%%%%%%%     Section2     %%%%%%%%%%%%%%%%%%%%%%%%
%%%%%%%%%%%%%%%%%%%%%%%%                      %%%%%%%%%%%%%%%%%%%%%%%%
%%%%%%%%%%%%%%%%%%%%%%%%%%%%%%%%%%%%%%%%%%%%%%%%%%%%%%%%%%%%%

\section{Proofs}

In this section we prove Theorem \ref{thm1}.
The best possible constant in \eqref{eq:1.1a}
obtained 
by Aubin~\cite{Au} and Talenti~\cite{Ta1} is
\begin{equation}\label{eq:1.1b}
S_p := \inf_{u \in W^{1,p}_0(B)} \frac{\ \|\nabla u\|_{L^p(B)}\ }{\|u\|_{L^{p^*}(B)}}
=
 \sqrt{\pi}N^{\frac1p}\left(\frac{N-p}{p-1}\right)^{\frac{p-1}{p}}
 \left[
  \frac{
    \Gamma\left(\frac{N}{p}\right)\Gamma\left(N+1-\frac{N}{p}\right)
   }
   {
   \Gamma(N)\Gamma\left(1+\frac{N}{2}\right)
   }
 \right]^{\frac1N}.
\end{equation}
This plays a crucial role.
First we define $\mathscr{B}_p$, $X_p$ by (\ref{B_p}), (\ref{X_p}), and then we set a constant $M_p$ by
\[
M_p := \sup_{\{u_n\} \in X_p} 
\left(\limsup_{n \to \infty} \int_B F_p(u_n)dx\right).
\]
With this setting, we divide the rest of the proof into two steps.
\begin{proposition}\label{proposition:2.1}
If $p>2N/(N+1)$ then
\begin{equation*}
\label{M_p}
M_p = |B| + \left[\frac{N-p}{N(p-1)}\a_p\right]^{\frac{N(p-1)}{N-p}} 
S_{p}^{-p^*}.
\end{equation*}
\end{proposition}

\begin{proposition}\label{proposition:2.2}
It holds that
\[
\lim_{p \to N} M_p = |B|(1+e^{\sum_{k=1}^{N-1} \frac{1}{k}}).
\]
\end{proposition}
Theorem~\ref{thm1} follows from Propositions~\ref{proposition:2.1} and \ref{proposition:2.2}.

\begin{proof}[Proof of Proposition~\ref{proposition:2.1}]
For any positive constants $a$, $b$ and $\gamma > 1$, it holds that
\begin{equation*}
a^\gamma + b^\gamma \leq (a + b)^\gamma \leq a^\gamma + b^\gamma + \gamma 2^{\gamma-1} \(ab^{\gamma-1} + a^{\gamma-1}b\).
\end{equation*}
The second inequality follows from the positivity of the function
\begin{equation*}
f_c(d)
:=
c^{\gamma} + d^{\gamma} + \gamma 2^{\gamma - 1}(cd^{\gamma-1} + c^{\gamma-1}d)
- (c + d)^{\gamma}
\end{equation*}
for fixed positive constant $c$ and any $d \in (0, c]$.
Thus, for $p > 2N/(N+1)$ we have that
\begin{equation}
\label{lowerest}
1+ \left[ \frac{N-p}{N(p-1)}\a_p\right]^{\frac{N(p-1)}{N-p}} 
|s|^{p^*} \leq F_p(s)
\end{equation}
and
\begin{equation}
\label{upperest}
F_p(s) \leq 1+ \left[ \frac{N-p}{N(p-1)}\a_p\right]^{\frac{N(p-1)}{N-p}} 
 |s|^{p^*} + H(s),
\end{equation}
where $H(s) = C_1 |s|^{\frac{p}{p-1}} + C_2 |s|^{p^*-\frac{p}{p-1}}$ with positive constants $C_1, C_2$.

We take any sequence $\{u_n\} \in X_p$.
We prove that
\begin{equation}
\label{Hzero}
\int_B H(u_n)dx = o(1)
\end{equation}
as $n \to \infty$.
To this end, we set
\[
\tau_1 := \int_B |u_n|^{\frac{p}{p-1}}dx, \quad \tau_2 := \int_B |u_n|^{p^*-\frac{p}{p-1}}d
x \quad {\rm and} \quad \kappa_{n, \e} = u_n|_{\partial B_\e}.
\]
Recall that the embedding 
\begin{equation}\label{eq:2.4a}
W^{1,p}_{rad}(B \setminus \ol B_\e) \hookrightarrow C^0\left(\ol {B \setminus B_\e}\right)
\end{equation}
holds for every $\e>0$,
and 
for every $q \in [1,p^*]$ there is a constant $S_q$ such that
\begin{equation}
\label{mixedbdSobolev}
S_q \| u \|_{L^{q}(B \setminus \ol B_\e)} \leq \| \nabla u \|_{L^p (B \setminus \ol B_\e)}
\end{equation}
for every $u \in W^{1,p}_{rad}\left(B \setminus \overline{B_\e}\right)$
with $u=0$ on $\partial B$.
By the definitions of $\mathscr{B}_p$ and $X_p$, we obtain $\| \nabla u_n \|_{L^p(B \setminus \ol B_\e)} = o(1)$, and then $\| u_n \|_{L^p(B \setminus \ol B_\e)} \to 0$ as $n \to \infty$ for any $\e>0$ by (\ref{mixedbdSobolev}).
It holds from the embedding~\eqref{eq:2.4a}
that
$\kappa_{n,\e} \to 0$ as $n \to \infty$ for any $\e>0$.
Using this fact and (\ref{mixedbdSobolev}) again, we observe that
\begin{eqnarray*}
\tau_1 &=& \int_{B_\e} |u_n|^{\frac{p}{p-1}}dx + \int_{B \setminus B_\e} |u_n|^{\frac{p}{p-1}}dx \\
&\leq& \( \| u_n - \kappa_{n,\e} \|_{L^{\frac{p}{p-1}}(B_\e)} + \| \kappa_{n,\e}\|_{L^{\frac{p}{p-1}}(B_\e)} \)^{\frac{p}{p-1}} + \int_{B \setminus B_\e} |u_n|^{\frac{p}{p-1}}dx \\
&=& \( \|u_n - \kappa_{n,\e} \|_{L^{p^*}(B_\e)}
|B_\e|^{\frac{(p-1)p^*-p}{(p-1)p^*}
} + o_n(1) \)^{\frac{p}{p-1}} + o_n(1)\\
&\leq& \( S_p^{-1}\|\nabla u_n\|_{L^p(B_\e)} |B_\e|^{\frac{(p-1)p^*-p}{(p-1)p^*}
} + o_n(1) \)^{\frac{p}{p-1}} + o_n(1) \\
&=& o_\e(1) +o_n(1),
\end{eqnarray*}
where $o_n(1) \to 0$ as $n \to \infty$ and $o_\e(1) \to 0$ as $\e \to 0$.
Letting $\e \to 0$ after $n \to \infty$, we obtain $\tau_1 = o(1)$ as $n \to \infty$.
Similarly, we deduce that $\tau_2 = o(1)$ as $n\to \infty$.
Thus, we obtain (\ref{Hzero}).

Applying (\ref{Hzero}) to (\ref{upperest}) with the aid of the Sobolev inequality,
we have
\begin{equation}
\begin{aligned}
\label{upperest2}
\int_B F_p(u_n) dx
&\leq \int_B \left\{ 1+ \left[ \frac{N-p}{N(p-1)}\a_p
\right]^{\frac{N(p-1)}{N-p}} 
 |u_n|^{p^*} + H(u_n)\right\} dx 
 \\
&\leq |B| + \left[ \frac{N-p}{N(p-1)}\a_p\right]^{\frac{N(p-1)}{N-p}} 
S_p^{-p^*} + o(1). 
\end{aligned}
\end{equation}
This proves the upper estimate of $M_p$.

It remains to prove the lower estimate of $M_p$.
We consider the Aubin--Talenti function
\[
U(x) = (1+|x|^{\frac{p}{p-1}})^{-\frac{N-p}{p}}
\]
and 
the modified Aubin--Talenti function
\begin{equation}\label{eq:Wn}
W_{n}(x) = K_n \left[ {\e_n}^{-\frac{N-p}{p}} \( U(x/{\e_n}) - U(1/{\e_n}) \) \right]
\end{equation}
for $\e_n \to 0$ as $n \to \infty$,
where 
$K_n:=1/\|\nabla U\|_{L^p(B_{1/\e_n})}$
so that
$\| \nabla W_n \|_{L^p(B)}=1$.
It is easy to see that $\{W_n\} \in X_p$.
Let us prove 
\begin{equation}\label{eq:2.8b}
\int_B |W_n|^{p^*} dx = S_p^{-p^*} + o(1)
\end{equation}
as $n\to \infty$.
First we 
recall that 
$S_p^{-1}=\frac{\|U\|_{L^{p^*}(\mathbb{R}^N)}}{\|\nabla U\|_{L^p(\mathbb{R}^N)}}$.
It follows from the definition of $W_n$ that
\[
\int_B |W_n|^{p^*} dx 
\le
\int_B \left|K_n \e_n^{-\frac{N-p}{p}}U\left(\frac{x}{\e_n}\right)\right|^{p^*}dx
=
\frac{\int_{B_{{1}/{\e_n}}} \left|U(y)\right|^{p^*}dy}{\|\nabla U\|^{p^*}_{L^p(B_{{1}/{\e_n}})}}
=
S_p^{-p^*} + o(1).
\]
On the other hand, the Taylor expansion yields
\[
\begin{aligned}
\int_B |W_n|^{p^*} dx
&
\ge 
\int_B \left|K_n \e_n^{-\frac{N-p}{p}}U\left(\frac{x}{\e_n}\right)\right|^{p^*}dx
\\
&
\qquad
-p^*
K_n^{p^*}
\e_n^{-\frac{N-p}{p}}U\left(\frac{1}{\e_n}\right)
\int_B
\left| \e_n^{-\frac{N-p}{p}}U\left(\frac{x}{\e_n}\right)\right|^{p^*-1}dx.
\end{aligned}
\]
Since 
$K_n\to \frac{1}{\|\nabla U\|_{L^p(\mathbb{R}^N)}}$,
$\e_n^{-\frac{N-p}{p}}U\left(\frac{1}{\e_n}\right)
\to 0$
as $n\to \infty$,
and
\[
\int_B
\left| \e_n^{-\frac{N-p}{p}}U\left(\frac{x}{\e_n}\right)\right|^{p^*-1}dx
\le |B|^{\frac{1}{p^*}}
\left(\int_{\mathbb{R}^N} |U(y)|^{p^*}dy\right)^{1-\frac{1}{p^*}}
\]
for all $n\in \mathbb{N}$, 
we have 
$\int_B |W_n|^{p^*} dx\ge S_p^{-p^*}+o(1).
$
Hence \eqref{eq:2.8b} holds.
Combining \eqref{eq:2.8b} with (\ref{lowerest}), we have
\begin{equation}
\begin{aligned}
\label{lowerest2}
&|B| + \left[ \frac{N-p}{N(p-1)}\a_p\right]^{\frac{N(p-1)}{N-p}} 
S_p^{-p^*} 
\\
&
= 
\lim_{n \to \infty} 
\int_B 
\left\{ 
1+ 
 \left[ 
   \frac{N-p}{N(p-1)}\a_p
 \right]^{\frac{N(p-1)}{N-p}} 
|W_n|^{p^*} \right\} dx \\
&\leq \sup_{\{u_n\} \in X_p} \limsup_{n \to \infty} \int_B F_p(u_n) dx. 
\end{aligned}
\end{equation}
This proves the lower estimate~\eqref{lowerest2}.
Consequently (\ref{upperest2}) and (\ref{lowerest2}) yield 
Proposition~\ref{proposition:2.1}.
\end{proof}

\begin{proof}[Proof of Proposition~\ref{proposition:2.2}]
It follows from \eqref{F_p}, \eqref{eq:1.1b}, and Proposition~\ref{proposition:2.1}
that
\begin{eqnarray*}
M_p 
&=&
|B| + \left[\frac{N-p}{N(p-1)}
\left(
|B|^{\frac{1}{p^*}} \frac{\sqrt \pi N}{\Gamma\(1 + \frac{N}{2}\)^{\frac{1}{N}}}
\right)^{\frac{p}{p-1}}
\right]^{\frac{N(p-1)}{N-p}} 
\\
&\ &
\times 
\left\{ \sqrt{\pi}N^{\frac1p}\left(\frac{N-p}{p-1}\right)^{\frac{p-1}{p}}
 \left[
  \frac{
    \Gamma\left(\frac{N}{p}\right)\Gamma\left(N+1-\frac{N}{p}\right)
   }
   {
   \Gamma(N)\Gamma\left(1+\frac{N}{2}\right)
   }
 \right]^{\frac1N}
\right\}^{-p^*} 
\nonumber \\
&=& 
|B| + |B| \left[ \frac{\Gamma(N)}{\Gamma \(\frac{N}{p}\) \Gamma \(N + 1 - \frac{N}{p}\)} \right]^{\frac{p}{N-p}}.
\nonumber
\end{eqnarray*}
We observe that
\begin{equation*}
\left[ \frac{\Gamma(N)}{\Gamma \(\frac{N}{p}\) \Gamma \(N + 1 - \frac{N}{p}\)} \right]^{\frac{p}{N-p}}
= \exp \left\{ \log \left[ \frac{\Gamma(N)}{\Gamma \(\frac{N}{p}\) \Gamma \(N + 1 - \frac{N}{p}\)} \right]^{\frac{p}{N-p}}\right\}.
\end{equation*}
Put $t = (N-p)/p$, then it is easy to see that
\[
\begin{aligned}
\lim_{p\to N}\log \left[ \frac{\Gamma(N)}{\Gamma \(\frac{N}{p}\) \Gamma \(N + 1 - \frac{N}{p}\)} \right]^{\frac{p}{N-p}}
&
=
\lim_{t\to 0}\frac{\log \left[ \Gamma(N) \right]-\log \left[ \Gamma \(t+1\)\Gamma \(N  - t\) \right]}{t}
\\
&
=
\Biggl.\frac{d}{dt}\Bigl[-\log \big(\Gamma(t+1)\Gamma(N-t)\bigr)\Bigr]\Biggr|_{t=0}
\\
&
=
\frac{d}{dt}\log (\Gamma(t)) \bigg|_{t=N} - \frac{d}{dt}\log (\Gamma(t)) \bigg|_{t=1}.
\end{aligned}
\]
Here 
$\frac{d}{dt}\log (\Gamma(t))=\frac{\Gamma'(t)}{\Gamma(t)}$
is called the digamma function.
It is known that the digamma function is written by (see for example Section 13.2 in \cite{AWH})
\[
\frac{d}{dt}\log (\Gamma(t)) = -\gamma + \sum_{j=1}^\infty \( \frac{1}{j} - \frac{1}{t-1+j}\),
\]
where $\gamma$ denotes Euler's constant.
Thus, 
it holds
\[
\lim_{p\to N}\log \left[ \frac{\Gamma(N)}{\Gamma \(\frac{N}{p}\) \Gamma \(N + 1 - \frac{N}{p}\)} \right]^{\frac{p}{N-p}} = \sum_{k=1}^{N-1}\frac{1}{k}.
\]
This completes the proof of Proposition~\ref{proposition:2.2}.
\end{proof}

We give a proof of \eqref{pradial} for the convenience of readers. 
\begin{proof}[Proof of \eqref{pradial}]
Let $1<p<N$ and fix $u\in W^{1,p}_{0,rad}(B)$.
Then there exists $v:[0,1)\to \mathbb{R}$ such that
$u(x)=v(|x|)$. 
By the fundamental theorem of calculus and the H\"older inequality, we have
\[
\begin{aligned}
|v(r)|
\le 
\int_r^1 |v'(s)|ds
&
\le
\left(
\int_r^1 s^{N-1}|u'(s)|^pds
\right)^{\frac{1}{p}}
\left(\int_r^1s^{-\frac{N-1}{p-1}}ds\right)^{\frac{p-1}{p}}
\\
&
\le
\omega_{N-1}^{-\frac{1}{p}}\|\nabla u\|_{L^p(B)}
\left\{
\frac{p-1}{N-p}\left(r^{-\frac{N-p}{p-1}}-1\right)
\right\}^{\frac{p-1}{p}}.
\end{aligned}
\]
The conclusion follows from 
$\omega_{N-1}^{-\frac{1}{p}}
=
N^{\frac{p-1}{p}}
\left(|B|^{\frac{1}{p}-\frac{1}{N}}\alpha_N^{\frac{N-1}{N}}\right)^{-1}$.
\end{proof}

\section{Additional remarks}
In a final section, we state some remarks. 
\begin{remark}
Proposition~\ref{proposition:2.1}, and hence Theorem \ref{thm1}, holds when 
$\mathscr{B}_p$ and $X_p$ are replaced by the following 
general settings, without assuming 
radially symmetric conditions:
\begin{equation*}
\begin{aligned}
\mathscr{C}_p 
&
:= \left\{ u \in W^{1,p}_0(B) \ \middle|\ \|\nabla u\|_{L^p(B)} \leq 1\right\},
\\
\hat X_p 
&
:= \left\{\{u_n\} \subset \mathscr{C}_p \ \middle|\ u_n \rightharpoonup 0 {\rm \ weakly \ in\ }W^{1,p}_0(B) \right\}.
\end{aligned}
\end{equation*}
We give a sketch of the proof. Since $\{W_n\}$ constructed in \eqref{eq:Wn} belongs to $\hat{X_p}$, 
the lower estimate of Proposition~\ref{proposition:2.1} holds by the same argument as in the proof for $X_p$. 
For the upper estimate, it is enough to prove $\intb H(u_n) dx = o(1)$ for any $\{u_n\} \in \hat X_p$, where $H$ is defined in \eqref{upperest}.
This computation is a direct consequence of the definition of $H$ and the compactness of subcritical Sobolev embeddings.
Hence, Proposition \ref{proposition:2.1} holds for $\hat X_p$.

\end{remark}

\begin{remark}
Several 
other 
inequalities for $W^{1,N}$ functions 
were 
derived from $W^{1,p}$ cases
by 
the direct limiting procedure as 
$p\to N$.
Indeed, $W^{1,p}$ approximation of the Alvino inequality~\eqref{Alvino_ineq}
and the Hardy inequality in the half space 
was obtained in \cite{Ioku} and \cite{ST}, respectively.
\end{remark}

\begin{remark}
It is worth to notice that
the function $F_p$ in Theorem~\ref{thm1} 
can be written by
the $q$-exponential function 
\[
\exp_q(r)
:=[1+(1-q)r]^{\frac{1}{1-q}},
\quad 
\text{for $q>0$, $q\neq 1$, and $r>0$,}
\]
which was
originally introduced by
Tsallis~\cite{Tsallis}
to study nonextensive statistics.
Under this notation, it is easy to check 
that
\begin{equation*}
F_p(u)=\exp_{1-\frac{N-p}{N(p-1)}}\left(\alpha_p|u|^{\frac{p}{p-1}}\right).
\end{equation*}
Since 
$\lim_{q\to 1}\exp_q r=e^r$, 
our functional defined by $F_p$ is regarded as a $q$-exponential approximation of the Moser--Trudinger functional.
\end{remark}

\begin{remark}
We recall that there exists a huge literature on the
whole space version of Moser--Trudinger inequality.
It would be impossible to list all the contributions, hence we like to single out some pioneering works in $\mathbb{R}^2$.
The whole space version of the Trudinger inequality is firstly obtained by Ogawa~\cite{Ogawa}.
Later, Adachi--Tanaka~\cite{AT} sharpened Ogawa's result by obtaining the best exponent.
Ruf~\cite{Ruf} pointed out that the inhomogeneous normalization by $\|\cdot\|_{H^1(\mathbb{R}^2)}$ makes differences in the case $\alpha=\alpha_2=4\pi$.
Recently Cassani--Sani--Tarsi~\cite{CST} showed a surprising equivalence between Adachi--Tanaka's inequality and Ruf's inequality.
It would be interesting to consider similar results to Theorem~\ref{thm1} for these inequalities.
Moreover,
extensions of Theorem~\ref{thm1} to, higher order derivative cases, fractional derivative cases, weighted versions, and for more general domains,
are possible future works.
\end{remark}

\begin{remark}
By 
Lions \cite{Lions}, it has been proven that if a sequence $\{u_n\} \subset \mathscr{C}_N$ satisfies $u_n \rightharpoonup u_0$ weakly in $W^{1,N}_0(B)$ and
\begin{equation*}
\liminf_{n \to \infty}\int_B e^{\a_N|u_n|^{\frac{N}{N-1}}}dx > \int_B e^{\a_N|u_0|^{\frac{N}{N-1}}}dx,
\end{equation*}
then 
$u_0=0$.
However, the situation 
in the case of
$p<N$ is different,
because
one can construct a sequence $\{u_n\} \subset \mathscr{C}_p$ 
such that
$u_n \rightharpoonup u_0 \not \equiv 0$ weakly in $W^{1,p}_0(B)$ 
and 
$\liminf_{n\to \infty}\int_B F_p(u_n) dx > \int_B F_p(u_0) dx$ as follows:
Let 
\begin{equation*}
T_p(s) = F_p(s) - \left[ \frac{N-p}{N(p-1)}\a_p\right]^{\frac{N(p-1)}{N-p}} 
|s|^{p^*}.
\end{equation*}
By (\ref{lowerest}) and (\ref{upperest}),
we observe that $1 \leq T_p(s) \leq 1 + H(s)$.
Applying a variant of the dominated convergence theorem, we have
\begin{equation*}
\int_B T_p(u_n) dx \to \int_B T_p(u_0) dx
\end{equation*}
for any $\{u_n\} \subset W^{1,p}_0 (B)$ with $u_n \rightharpoonup u_0$ weakly in $W^{1,p}_0(B)$.
Therefore, it suffices to identify a sequence $\{u_n\} \subset \mathscr{C}_p$
such that 
$u_n \rightharpoonup u_0 \not \equiv 0$ weakly in $W^{1,p}_0(B)$ 
and 
$\liminf_{n\to \infty}\int_B |u_n|^{p^*} dx > \int_B |u_0|^{p^*} dx$.

Take $\phi, \psi \in W^{1,p}_0(B)$ with 
$\|\nabla \phi \|_{L^p(B)}^p = 1/2, \|\nabla \psi \|_{L^p(B)}^p = 1/2$
and consider zero extension of $\psi$ outside of $B$.
Define a sequence by
\[
u_n(x) = C_n \(\phi(x) + n^{\frac{N-p}{p}}\psi (nx)\),
\]
where $C_n$ is taken such that 
$\| \nabla u_n \|_{L^p(B)}=1$.
Under the setting, 
it holds that
\[
\begin{aligned}
1 &= C_n^p 
\| \nabla \phi \|_{L^p(B \setminus B_{1/n})}^p 
+ 
\| \nabla u_n \|_{L^p(B_{1/n})}^p \\
&= C_n^p
\left\{ 
\frac{1}{2} 
+ 
 \left[
	\|\nabla n^{\frac{N-p}{p}} \psi (n \cdot)\|_{L^p(B_{1/n})} 
	+
	O\(\| \nabla \phi \|_{L^p(B_{1/n})}\)
 \right]^p
+o(1)
\right\} \\
&=
C_n^p
\(
\frac{1}{2} 
+ 
\|\nabla \psi \|_{L^p(B)}^p
+o(1)
\) 
=
C_n^p
\(1+ o(1)\).
\end{aligned}
\]
Thus, we see that
$C_n \to 1$ and $u_n \rightharpoonup \phi$ weakly in $W^{1,p}_0(B)$
as $n \to \infty$. Therefore, $\{u_n\}$ does not belong to $\hat{X_p}$.
Moreover, 
it follows from 
the Brezis--Lieb Lemma
\cite[Theorem 1]{BL}
that
\begin{eqnarray*}
\lim_{n \to \infty} \int_B |u_n|^{p^*} dx 
&=& 
\int_B |\phi|^{p^*} dx + \lim_{n \to \infty}\int_B \bigl|n^{\frac{N-p}{p}}\psi (n x) \bigl|^{p^*}dx \\
&=& 
\int_B |\phi|^{p^*} dx + \int_B |\psi|^{p^*}dx
>
\int_B |\phi|^{p^*} dx,
\end{eqnarray*}
hence 
the sequence $\{u_n\}$ satisfies the desired condition.
This fact makes difficult to prove the attainability of the optimal constant
$
\sup_{\|\nabla u\|_{L^p(B)}\le 1}\int_B F_p(u)dx.
$
\end{remark}

\begin{remark}
The optimal constant
$
\sup_{\|\nabla u\|_{L^p(B)}\le 1}\int_B F_p(u)dx
$
is lower semicontinuous as $p\to N$, namely there holds
\[
\liminf_{p\uparrow N}\left(\sup_{\|\nabla u\|_{L^p(B)}\le 1}\int_B F_p(u)dx \right)
\ge
\sup_{\|\nabla u\|_{L^N(B)}\le 1}\int_B e^{\alpha_N|u|^{\frac{N}{N-1}}}dx.
\]
Indeed,
it follows from
$\|\nabla u\|_{L^p(B)}\le |B|^{\frac{1}{p}-\frac{1}{N}}\|\nabla u\|_{L^N(B)}$
that
\[
\begin{aligned}
&
\sup_{\|\nabla u\|_{L^p(B)}
\le 
1
}\int_B F_p(u)dx
\\
&
=
\sup_{\|\nabla u\|_{L^p(B)}\le 
|B|^{\frac{1}{p}-\frac{1}{N}}
}
\int_B 
\left[1 + \frac{N-p}{N(p-1)}\a_N^{\frac{p(N-1)}{N(p-1)}} 
|u|
^{\frac{p}{p-1}}
\right]^{\frac{N(p-1)}{N-p}}
dx
\\
&
\ge
\sup_{\|\nabla u\|_{L^N(B)}\le 1
}
\int_B 
\left[1 + \frac{N-p}{N(p-1)}\a_N^{\frac{p(N-1)}{N(p-1)}} |u|^{\frac{p}{p-1}}
\right]^{\frac{N(p-1)}{N-p}}
dx
\\
&
\to 
\sup_{\|\nabla u\|_{L^N(B)}\le 1}\int_B e^{\alpha_N|u|^{\frac{N}{N-1}}}dx
\ \ 
(p\to N).
\end{aligned}
\]
The lower semicontinuity
gives us an alternative proof of Moser--Trudinger inequality~\eqref{originalTM}
if one can prove 
the uniform bound on $p$
of the optimal constant
$
\sup_{\|\nabla u\|_{L^p(B)}\le 1}\int_{B}F_p(u) dx
$.
The uniform bound and
the continuity 
of the optimal constant with respect to $p$
 remain open.
\end{remark}


\begin{thebibliography}{99}

\bibitem{AT}
S. Adachi, K. Tanaka,
{\it Trudinger type inequalities in $\mathbb{R}^N$ and their best exponents},
Proc. Amer. Math. Soc.
{\bf 128} (1999), no. 7, 2051--2057.

\bibitem{Adams}
R. A. Adams, 
{\it On the Orlicz--Sobolev imbedding theorem},
 J. Funct. Anal. {\bf 24} (1977), 241--257.

\bibitem{Alvino} A. Alvino,
{\it Sulla diseguaglianza di Sobolev in spazi di Lorentz},
Boll. Un. Mat. Ital. A (5) {\bf 14} (1977), no. 1, 148--156.

\bibitem{AWH} G. B. Arfken, H. J. Weber, F. E. Harris,
Mathematical methods for physicists, Seventh edition,
Elsevier Inc., 2013, ISBN: 978-0-12-384654-9

\bibitem{Au}
T. Aubin, 
{\it Probl\` emes isop\'erim\'etriques et espaces de Sobolev},
J. Differ. Geom. {\bf 11} (1976), 573--598.

\bibitem{BL} H. Brezis, E. Lieb,
{\it A relation between pointwise convergence of functions and convergence of functionals},
Proc. Amer. Math. Soc. {\bf 88} (1983), no. 3, 486--490.

\bibitem{CC} L. Carleson, S.-Y. A. Chang,
{\it On the existence of an extremal function for an inequality of J. Moser},
Bull. Sci. Math. (2) {\bf 110} (1986), no. 2, 113--127.

\bibitem{Cianchi}
A. Cianchi,
{\it A sharp embedding theorem for Orlicz--Sobolev spaces}, Indiana Univ. Math. J. {\bf 45} (1996), 39--65.

\bibitem{CST}
D. Cassani, F. Sani, C. Tarsi,
{\it Equivalent Moser type inequalities in $\re^2$ and the zero mass case},
J. Funct. Anal. {\bf 267} (2014), no. 11, 4236--4263.

\bibitem{FOR} D. G. de Figueiredo, J. M. do \'O, B. Ruf,
{\it On an inequality by N. Trudinger and J. Moser and related elliptic equations},
Comm. Pure Appl. Math. {\bf 55} (2002), no. 2, 135--152.

\bibitem{Flucher} M. Flucher,
{\it Extremal functions for the Trudinger-Moser inequality in 2 dimensions},
Comment. Math. Helv. {\bf 67} (1992), no. 3, 471--497.

\bibitem{Hashizume}
M. Hashizume,
{\it Maximization problem on Trudinger-Moser inequality involving Lebesgue norm},
J. Funct. Anal. {\bf 279} (2020), no. 2, 108513, 30 pp.

\bibitem{IMNS}
S. Ibrahim, N. Masmoudi, K. Nakanishi, F. Sani,
{\it Sharp threshold nonlinearity for maximizing the Trudinger--Moser inequalities},
J. Funct. Anal. {\bf 278} (2020), no. 1, 108302, 52 pp. 

\bibitem{Ioku} N. Ioku,
{\it Attainability of the best Sobolev constant in a ball},
Math. Ann. {\bf 375} (2019), no. 1--2, 1--16.

\bibitem{Ishiwata}
M. Ishiwata, 
{\it Existence and nonexistence of maximizers for variational problems associated with Trudinger--Moser type inequalities in $\mathbb{R}^N$},
Math. Ann.{\bf 351} (2011), no. 4, 781--804.

\bibitem{Li} Y. Li,
{\it Extremal functions for the Moser--Trudinger inequalities on compact Riemannian manifolds},
Sci. China Ser. A {\bf 48} (2005), no. 5, 618--648.

\bibitem{Lions} P.-L. Lions,
{\it The concentration-compactness principle in the calculus of variations. The limit case. I},
Rev. Mat. Iberoamericana {\bf 1} (1985), no. 1, 145--201.

\bibitem{MM}
G. Mancini, L. Martinazzi,
{\it The Moser-Trudinger inequality and its extremals on a disk via energy estimates},
Calc. Var. Partial Differential Equations {\bf 56} (2017), no. 4, Paper No. 94, 26 pp.

\bibitem{Moser} J. Moser,
{\it A sharp form of an inequality by N. Trudinger},
Indiana Univ. Math. J. {\bf 20} (1970/71), 1077--1092.

\bibitem{Ogawa}
T. Ogawa, 
{\it A proof of Trudinger's inequality and its application to nonlinear
Schr\"odinger equations}, 
Nonlinear Anal. {\bf 14} (1990), 765--769.

\bibitem{Pohozaev}
S. I. Pohozaev, 
{\it 
The Sobolev Embedding in the Case $pl = n$}, 
Proc. Tech. Sci. Conf. on
Adv. Sci. Research 1964--1965, Mathematics Section, Moskov. \`Energet. Inst. Moscow, 1965,
158--170.

\bibitem{Ruf}
B. Ruf, 
{\it A sharp Trudinger--Moser type inequality for unbounded domains in $\mathbb{R}^2$}, J. Funct. Anal. {\bf 219} (2005), no. 2, 340--367.

\bibitem{ST}
M. Sano, F. Takahashi,
{\it Critical Hardy inequality on the half-space via the harmonic transplantation},
Calc. Var. Partial Differential Equations {\bf 61}, Article number:158 (2022)

\bibitem{Struwe} M. Struwe,
{\it Critical points of embeddings of $H^{1,n}_0$ into Orlicz spaces},
Ann. Inst. H. Poincar\'e Anal. Non Lin\'eaire {\bf 5} (1988), no. 5, 425--464.

\bibitem{Ta1}
G. Talenti, 
{\it Best constant in Sobolev inequality}, Ann. Mat. Pura Appl. {\bf 110} (1976), 353--372.

\bibitem{Trudinger} N. S. Trudinger,
{\it On imbeddings into Orlicz spaces and some applications},
J. Math. Mech. {\bf 17} 1967 473--483.

\bibitem{Tsallis}
C. Tsallis,
{\it Possible generalization of Boltzmann-Gibbs statistics},
J. Statist. Phys. {\bf 52} (1988), %no. 1-2, 
479--487.

\bibitem{Yudovich}
V. I. Yudovich,
{\it 
Some estimates connected with integral operators and
with solutions of elliptic equations},
 Dokl. Akad. Nauk SSSR {\bf 138} (1961)
805--808. Translated in Soviet Math. Dokl. {\bf 2} (1961), 746--749.

\end{thebibliography}
\end{document}